\documentclass[11pt,reqno]{amsart}
\usepackage{amssymb,amsmath}
\usepackage{amsmath,amstext,amssymb,amscd}
\usepackage{verbatim}
\usepackage{enumerate}
\usepackage{mathrsfs}
\usepackage[dvipsnames,usenames]{xcolor}

\usepackage{amsthm}

\usepackage{color}
\usepackage{setspace}

\newcommand{\dist}{{\rm dist}}

\def\Bbb{\mathbb}
\def\Cal{\mathcal}
\def\Dt{\partial_t}

\def\eb{\varepsilon}

\def\R {\mathbb{R}}
\def\N {\mathbb{N}}

\def\<{\left<}
\def\>{\right>}

\def\Nx{\nabla_x}
\def\Dx{\Delta_x}
\def\({\left(}
\def\){\right)}

\newtheorem{proposition}{Proposition}[section]
\newtheorem{theorem}[proposition]{Theorem}
\newtheorem{corollary}[proposition]{Corollary}
\newtheorem{lemma}[proposition]{Lemma}

\theoremstyle{definition}
\newtheorem{definition}[proposition]{Definition}

\newtheorem{remark}[proposition]{Remark}

\numberwithin{equation}{section}

\def \no#1#2#3 {{\bf #1} (#3), #2.}
\def \eds#1#2#3 {#1, #2, #3.}

\title[]
{Kolmogorov $\eb$-entropy of the uniform attractor for a wave equation}
\author[Y. Xiong and C. Sun] {Yangmin Xiong and Chunyou Sun${}^\ast$\\
\tiny School of Mathematics and Statistics, Lanzhou University \\ Lanzhou, 730000, P.R. China}
\begin{document}
\begin{abstract}
This paper is concerned with a non-autonomous sup-cubic semilinear wave equation in a smooth bounded domain of $\mathbb R^{3}$, using the introduced weak topology entropy, we obtain an upper bound for the $\varepsilon$-entropy of the
uniform attractor for the case where the external forces 
are not translation-compact. 

\end{abstract}




\thanks{$^*$ Corresponding author}
\thanks{Email address: xiongym18@lzu.edu.cn (Y.Xiong), sunchy@lzu.edu.cn (C.Sun)}
\thanks{This work was supported by the NSFC (Grant No. 12271227).}
\maketitle

\section{Introduction}\label{sec1}
We consider the following non-autonomous damped wave equation
\begin{equation}\label{eq1.1}
\begin{cases}
\partial_{t}^{2}u+\gamma\partial_{t}u-\Delta u+f(u)=g(x,t),~~x\in\Omega,~~t\geq\tau,\\
u\big|_{\partial\Omega}=0,\ \ \ (u,\partial_{t}u)\big|_{t=\tau}=(u_{\tau},u'_{\tau}),
\end{cases}
\end{equation}
where $\Omega$ is a bounded domain in
$\mathbb{R}^{3}$ with smooth boundary,
$\gamma>0$ is a fixed dissipation rate, $\Delta$ is the Laplacian with respect to variable $x$. The nonlinearity $f\in C^{1}(\mathbb{R})$ has a sup-cubic growth rate, i.e.
\begin{equation}\label{eq1.2}
|f'(s)|\leq C(1+|s|^{p-1}),~~1\leq p<5,~~\forall~s\in\mathbb{R},
\end{equation}
and satisfies the 
dissipative assumptions:
\begin{equation}\label{eq1.3}
f(s)s\geq-C, \ \ s\in\mathbb R,
\end{equation}
and there exists $\alpha>0$ such that
\begin{equation}\label{eq1.4}
F(s)\le \alpha f(s)s+C,
\end{equation}
where~$F(s)=\int_{0}^{s}f(r)dr$. The function $g(\cdot,s)$ is the given external force which is assumed translation bounded in $L_{loc}^{2}(\mathbb{R};L^{2}(\Omega))$, i.e.
\begin{equation}\label{eq1.5}
\|g\|_{L^2_b(\R;L^{2})}:=\sup_{t\in\R}\|g\|_{L^2(t,t+1;L^{2})}<\infty.
\end{equation}
The initial data $\xi_{u}(\tau)=(u_{\tau},u'_{\tau}):=\xi_{\tau}$ is taken from the standard energy space $\mathcal{E}:=H_{0}^{1}(\Omega)\times L^{2}(\Omega)$ with energy norm
\begin{equation*}
\|\xi_{u}\|_{\mathcal{E}}^{2}:=\|(u,\partial_{t}u)\|_{\mathcal{E}}^{2}=\|\nabla u\|^{2}+\|\partial_{t}u\|^{2},
\end{equation*}
here $\|\cdot\|$ is the usual $L^{2}(\Omega)$-norm.

The long-time behavior of the autonomous systems can be described by global attractors, which have been extensively studied in \cite{BV92,CV02,Chueshov15,Hale87,T97}, and this attractor usually has finite fractal and Hausdorff dimension.

The study becomes more complicated for the non-autonomous equations. One of the main methods is so-called uniform attractor, which originated from the work of Haraux \cite{Haraux88,Haraux91} and further developed by Chepyzhov and Vishik \cite{C13,CV94,CV02}.
By means of constructing skew product flow one
can reduce the problem to an autonomous one in the extended phase space and keep the invariance, but need the hull has some compactness.

The Hausdorff and fractal dimension of the uniform attractor may be infinite \cite{CV94,CV02,MZ08}.
A possible way to measure the "thickness" of infinite dimensional attractors,
which has been suggested in \cite{CV02},
is to estimate their
Kolmogorov~$\varepsilon$-entropy. 
Based on the results of the existence and structure of uniform attractor $\Cal{A}_{\Sigma}$, and under the assumptions that the symbol space $\Sigma$ is compact in the Hausdorff topological space $\Xi$ endowed with the local uniform convergence topology, the family of processes $\{U_{h}(t,\tau),h\in\Sigma\}$ satisfies the Lipschitz condition with respect to symbols and is uniformly quasidifferentiable, the authors in \cite{CV02} proved that the upper estimate of $\varepsilon$-entropy of uniform attractor has the following form
\begin{equation*}
\mathbb{H}_{\varepsilon}(\Cal{A}_{\Sigma};\mathcal{E})\leq D\log_{2}\frac{\eb_0}{\varepsilon}
 +\mathbb{H}_{\lambda\varepsilon}(\Sigma_{0,L\log_{2}\frac{\eb_0}{\varepsilon}};\Xi_{0,L\log_{2}\frac{\eb_0}{\varepsilon}}), \ \ \forall~\varepsilon\leq\varepsilon_{0}
\end{equation*}
provided that the behavior of entropy of $\Pi_{0,l}\Sigma=\Sigma_{0,l}$ in $\Xi_{0,l}$ is known. Here the positive constants $D$, $\lambda$, $L$ and $\varepsilon_{0}$ depend only on the parameters of equation and the topological space $\Xi$ where the symbol space $\Sigma$ live in, but are independent of $\eb$.

With the progress on the wave equations obtained in recent years, there are some results show that one can still obtain the existence and structure of the strong uniform attractor for the system with more general external forces which are not translation compact, see e.g. \cite{MZ07,MS19,SZ20,Zelik15}. Meanwhile, based on the recent extension of the Strichartz type estimates for the bounded domains in \cite{BSS09,BLP08}, the attractor theory has been developed for semilinear wave equations concerning sup-cubic nonlinearities in both autonomous and non-autonomous cases, see \cite{KSZ16,MSSZ21,MS19,SZ20}.

Based on the above existing results, it is natural to study the Kolmogorov $\varepsilon$-entropy of the uniform attractors for equation \eqref{eq1.1} with sup-cubic nonlinearity and more general external forces. Since for the quintic wave equations, it is still not clear the behavior of the energy-to-Strichartz estimate as $t\rightarrow\infty$ (crucial for the attractor theory) in bounded domain with Dirichlet boundary conditions,
so we just consider the sub-quintic case in this paper. And in particular, we consider the following two typical cases of extra regularity for $g$:
\begin{equation}\label{gt-reg}
g\in H_{b}^{1}(\mathbb{R};L^{2}(\Omega))
\end{equation}
or
\begin{equation}\label{gs-reg}
g\in L_{b}^{2}(\mathbb{R};H^{1}(\Omega)), 
\end{equation}
which belong to the so-called time regular or space regular functions in $L_{b}^{2}(\mathbb{R};L^{2}(\Omega))$ respectively, see \cite{Zelik15}.


As we know, the value of the $\varepsilon$-entropy may depend on the topology chosen. In contrast to the translation compact case, it would brings some technical difficulties in our case.
We note that the Lipschitz condition with respect to symbols proposed in \cite{CV02} or the derivation of some estimates for difference between solutions \cite{EMZ03,Zelik01} is of fundamental significance for estimate the~$\varepsilon$-entropy,
however, is no longer suitable in our case, at least it hard to verify in a straightforward way. Moreover, the fact that in weak topology the distance between two symbols on each subinterval cannot be controlled by the whole interval, which leads us cannot use the method of iterated.

To this end, we introduce a new number $\mathbb{H}_{\varepsilon}^w(A,X;\,B,Y)$, the weak topology entropy of the set $B$ in $Y$ corresponding to the set $A$ in $X$, which depends only on the spaces $X$, $Y$ and the sizes of $A$ and $B$, for our application turns to depend only the bounds of symbols and the parameters of equation \eqref{eq1.1}. This new definition is based on an observation that 
by some asymptotic smoothing property of solutions to verify that the smoothness of uniform attractors, 
then the collection of the second ingredient of the difference of solutions (with same initial data and different symbols $h_{1},h_{2}\in\Sigma$) is uniformly bounded in $X:=L^{\infty}(iT,(i+1)T;H_0^{1}(\Omega))\cap H^{1}(iT,(i+1)T;H^{-1}(\Omega))$ for all $i\in\mathbb{N}^{+}$ and for some $T>0$. Moreover, $X$ can be compactly embedded into $Y^{\ast}$, and $Y=L^{2}(iT,(i+1)T;L^{2}(\Omega))$ is the space where the symbol space $\Sigma_{[iT,(i+1)T]}$ lives in. Based on this, a special cover of $\Sigma_{[T,(l+1)T]}$, $l\in\mathbb{N}^{+}$ is defined and we obtain an upper bound of the number of this cover. Meanwhile, we construct an intermediate object based on the non-autonomous exponential attractor proposed in \cite{EMZ00,EMZ05}, which has
finite fractal dimension and satisfies uniform forward (and also pullback) exponential attracting, and the most important is that these two properties are independent of the specified choice of symbols that obtained by time translations of the initial symbol.

As an application, for the uniform attractor $\Cal{A}_{\Sigma}$ of equation \eqref{eq1.1}, under the assumptions \eqref{eq1.2}-\eqref{eq1.4} and $g$ satisfies either \eqref{gt-reg} or \eqref{gs-reg},
we obtain the following estimate of its $\varepsilon$-entropy
\begin{equation*}
\mathbb{H}_{\varepsilon}(\Cal{A}_{\Sigma};\mathcal{E})\leq D'\log_{2}\frac{1}{\varepsilon}+L'\log_{2}\frac{1}{\varepsilon}\mathbb{H}_{\lambda'\varepsilon^{\ell}}^{w}(r,X;\,\|g\|_{L_{b}^{2}(\mathbb{R};L^{2}(\Omega))},Y),
\end{equation*}
where the positive constants $D'$, $L'$, $\lambda'$, $\ell$ and $r$ depend only on
the spaces $X$, $Y$ and the parameters of equation \eqref{eq1.1}, but are independent of $\eb$.
We remark that the space $H_{b}^{1}(\mathbb{R};L^{2}(\Omega))$ or $L_{b}^{2}(\mathbb{R};H^{1}(\Omega))$ is much better than $L_{b}^{2}(\mathbb{R};L^2(\Omega))$,  however, a function just belongs to such more regular spaces 
still not good enough to deduce its hull is translation compact in $L_{loc}^{2}(\mathbb{R};L^2(\Omega))$ (e.g., see \cite{CV02,Zelik15}). Here, in this paper, we assume $g\in H_{b}^{1}(\mathbb{R};L^{2}(\Omega))$ or $L_{b}^{2}(\mathbb{R};H^{1}(\Omega))$ mainly for simplifying the calculations and presenting some examples to illustrate how to using the weak topology entropy to give an upper bound about the $\varepsilon$-entropy of the considered uniform attractor.

The rest of the paper is organized as follows. In Section 2, we give the symbol spaces and recall briefly the results about uniform attractors. In Section 3, we introduce an entropy $\mathbb{H}_{\varepsilon}^w(A,X;\,B,Y)$ with respect to the weak topology of $Y$ for the subsets-spaces pair $(A,X;\,B,Y)$. In Section 4, we verify that the uniform attractor of equation \eqref{eq1.1} is more regular if the external force $g\in L_b^{2}(\R;L^{2}(\Omega))$ is more regular. In Section 5, we derive some asymptotic smoothing estimates for difference of solutions of \eqref{eq1.1} that allows us to construct a family of discrete exponential attractor, which has uniform finite fractal dimension and enjoys uniform exponential attracting property.
Finally, an upper bound of the entropy of the attractor for equation \eqref{eq1.1} is given in Section 6 by using the new defined weak topology entropy.
\section{Preliminaries}\label{sec2}
Let $\Xi$ be a Hausdorff topological space and let $\{T(t),t\in\R\}$ be the continuous translation operators on $\Xi$:
$$
T(t)h(s):=h(s+t), \ \ s\in\R.
$$
Let all time-dependent coefficients of the considered equations belonging to a set $\Sigma$ with topology from $\Xi$. The set $\Sigma$ is called the symbol space and $h\in\Sigma$ is the symbol.

Let us consider a typical symbol space $\Sigma$ which only contains time-dependent external
forces. Let $H$ be a reflexive Banach space and we consider functions $g:\R\to H$. Assume that $g\in L_{b}^{p}(\mathbb{R};H)$, $1<p<\infty$. Consider the following set
\begin{equation*}
\Sigma_{0}:=\{T(t)g(s);t\in\R\}:=\{g(t+s);t\in\R\}
\end{equation*}
and the closure in $\Xi$ of the set $\Sigma_0$. This closure is said to be the hull of the function $g(s)$ in $\Xi$ and is denoted by $\Cal H(g)$.

For $g\in L_b^{p}(\R;H)$, $1<p<\infty$, due to the Banach-Alaoglu theorem, the hull
\begin{equation}\label{0.hull}
\Cal H(g):=[\{T(t)g,t\in\R\}]_{L^{p,w}_{loc}(\R;H)}
\end{equation}
is compact in $L^{p,w}_{loc}(\R;H)$. We denote by $L^{p,w}_{loc}(\mathbb{R};H)$ the space $L^{p}_{loc}(\mathbb{R};H)$ endowed with local weak convergence topology, which means a sequence $h_{n}$ converges to $h$ as $n\rightarrow\infty$ in $L_{loc}^{p,w}(\mathbb{R};H)$ if and only if
\begin{equation*}
\int_{t_{1}}^{t_{2}}\langle v(s),h_{n}(s)-h(s)\rangle ds\rightarrow0\ \ \text{as}~n\rightarrow\infty
\end{equation*}
for each bounded interval $[t_{1},t_{2}]\subset\mathbb{R}$ and any $v\in L^{q}(t_{1},t_{2};H^{\ast})$ with $1/p+1/q=1$. 
We know that $\Cal H(g)$ is compact in $L^{p,w}_{loc}(\R;H)$ and the translation group $\{T(t),t\in\R\}$ is continuous in the topology of $L_{loc}^{p,w}(\mathbb{R};H)$. It is also not difficult to see that
$$
\|h\|_{L^p_b(\R;H)}\le\|g\|_{L^p_b(\R;H)},\ \ \forall h\in\Cal H(g).
$$
Since translation boundedness is not sufficient to gain the existence of a strong uniform attractor,
the most natural and most studied is the class of translation-compact (tr.c.) external forces introduced by Vishik and Chepyzhov, see \cite{CV02}. We recall that $g\in L^p_b(\R;H)$ is translation-compact if the set $\Sigma_{0}$ is precompact in $L_{loc}^{p}(\mathbb{R};H)$.
In \cite{Zelik15}, the author proposed several more general classes of external forces which are not translation compact but can still establish the existence of strong uniform attractors. Here, we are mainly interested in the following two classes of external forces.
\begin{definition}[\cite{Zelik15}]
Let $H$ be a reflexive Banach space and $1<p<\infty$. A function $g\in L_{b}^{p}(\mathbb{R};H)$ is space regular if for every $\varepsilon>0$ there exists a finite-dimensional subspace $H_{\varepsilon}\subset H$, dim$H_{\varepsilon}<\infty$ and a function $g_{\varepsilon}\in L_{b}^{p}(\mathbb{R};H_{\varepsilon})$ such that $\|g-g_{\varepsilon}\|_{L_{b}^{p}(\mathbb{R};H)}\leq\varepsilon$. Analogously, a function $g\in L_{b}^{p}(\mathbb{R};H)$ is time regular if for every $\varepsilon>0$ there exists a function $g_{\varepsilon}\in H_{b}^{k}(\mathbb{R};H)$ for all $k>0$ such that $\|g-g_{\varepsilon}\|_{L_{b}^{p}(\mathbb{R};H)}\leq\varepsilon$.
\end{definition}

Now, let us recall some concepts and results related to the uniform attractor theory, see \cite{CV02,SZ20,Zelik15}.

Let $\mathcal E$ be a Hausdorff topological space. 
Let $\{U_h(t,\tau),h\in\Sigma\}$ be a family of dynamical processes on $\mathcal E$, that is, for each $h\in\Sigma$,  the two parameter family of operators $U_h(t,\tau)$ from $\Cal E$ to $\Cal E$ satisfy
$$
 U_h(\tau,\tau)=\operatorname{Id},\ \ U_h(t,\tau)=U_h(t,s)\circ U_h(s,\tau),\ \ t\ge s\ge\tau\in\mathbb R.
$$
Also, let $\mathbb B$ be a family of sets $B\subset\mathcal E$ such that if $B\in\mathbb B$ and $B_1\subset B$, then $B_1\in\mathbb B$. The sets $B\in\mathbb B$ are said to be bounded.

\begin{definition}[\cite{CV02}]
A set $\Cal{A}_{\Sigma}\subset\Cal E$ is called a uniform attractor for the family of processes $\{U_h(t,\tau),h\in\Sigma\}$ if:

$1)$~$\Cal{A}_{\Sigma}$ is compact and bounded in $\mathcal E$;

$2)$~$\Cal{A}_{\Sigma}$ is a uniformly attracting set for the processes $\{U_h(t,\tau),h\in\Sigma\}$, that means, for every $B\in\mathbb B$ and every neighborhood $\mathcal O(\Cal{A}_{\Sigma})$, there exists a time $T=T(\mathcal O,B)$ such that
$$
\cup_{h\in\Sigma}U_h(t,\tau)B\subset\mathcal O(\Cal{A}_{\Sigma}), \ \ t-\tau\geq T, \ \tau\in\mathbb R;
$$

$3)$~$\Cal{A}_{\Sigma}$ is a minimal set with the properties $1)$ and $2)$.
\end{definition}
Below, $\mathcal E$ will usually be a Banach space or even a Hilbert space endowed with either the strong or the weak topology. The associated uniform attractor will be referred to as a strong or a weak uniform attractor respectively. In both cases $\mathbb B$ consists of all bounded sets in the Banach space under consideration.

We recall the following standard existence theorem, see \cite{CV02} for details.
\begin{theorem}\label{exist thm}
Let $\Cal E$ and $\Xi$ be Hausdorff topological spaces and let $\Sigma$ be a compact set in the space $\Xi$. Let the family of processes $\{U_h(t,\tau),h\in\Sigma\}$ acting on $\mathcal E$ possesses a uniformly absorbing set $\mathcal B\in\mathbb B$ and is uniformly asymptotically compact on $\Cal B$. Then this family of processes has a uniform attractor $\Cal{A}_{\Sigma}\subset\mathcal B$.

Assume, in addition, that the map $(\xi,h)\rightarrow(U_h(t,\tau)\xi,T(t)h)$ is continuous for every fixed $t$ and $\tau$. Then $\Cal{A}_{\Sigma}$ possesses the following description:
\begin{equation}\label{1.rep}
\Cal A_{\Sigma}:=\cup_{h\in\Sigma}\Cal K_h,
\end{equation}
where
$$
\Cal K_h:=\{u:\R\to \Cal E,\ \ U_h(t,\tau)u(\tau)=u(t),\ t\ge\tau\in\R\}
$$
is a set of all complete bounded trajectories of the process $U_h(t,\tau)$ (the so-called kernel of $U_h(t,\tau)$ in the terminology of Vishik and Chepyzhov, see \cite{CV02}).
\end{theorem}

We now turn to the damped wave equation \eqref{eq1.1}.
Recall some known results related with the non-autonomous system \eqref{eq1.1} and the corresponding
uniform attractors, see e.g., \cite{MS19,SZ20,Zelik15}.

We start with the following linear wave equation
\begin{equation}\label{linear}
\Dt^{2}v+\gamma\Dt v-\Delta v=G(t), \ \ \xi_v\big|_{t=\tau}=\xi_\tau, \ \ v\big|_{\partial\Omega}=0.
\end{equation}
\begin{theorem}[\cite{SZ20}]
Let the initial data $\xi_\tau\in\Cal E^{\alpha}$ and $G\in L_{loc}^{1}(\R;\Cal E^{\alpha})$ for some $\alpha\in\R$. Then there exist a unique solution $\xi_v\in C_{loc}(\R;\Cal E^{\alpha})$ of the problem \eqref{linear}. Also, the solution $v$ belongs to the space $L_{loc}^{4}(\R;H^{\alpha,12}(\Omega))$. Moreover, the following estimate holds:
\begin{multline}\label{linenergyest}
\|\xi_v(t)\|_{\Cal E^{\alpha}}+\(\int_\tau^{t}e^{-4\beta(t-s)}\|v(s)\|_{H^{\alpha,12}}^{4}\,ds\)^{1/4}\le \\ \le C\|\xi_\tau\|_{\Cal E^{\alpha}}e^{-\beta(t-\tau)}
+C\int_\tau^{t}e^{-\beta(t-s)}\|G(s)\|_{H^{\alpha}}\,ds,
\end{multline}
where the positive constants $C$ and $\beta$ are independent of $t\geq\tau$, $G$ and $\xi_\tau$.
\end{theorem}

The well-posedness and dissipativity of problem \eqref{eq1.1} is verified in \cite{MS19,SZ20} in the class of the so-called Shatah-Struwe (SS) solutions, see \cite{BSS09,BLP08,SS93}. We recall that $u(t)$ is a SS-solution of problem \eqref{eq1.1} if $\xi_u\in C([\tau,T],\mathcal E)$ and
\begin{equation*}\label{2.str}
u\in L^4(\tau,T; L^{12}(\Omega))\ \ \text{for all}\ \ T>\tau
\end{equation*}
and if it satisfies equation \eqref{eq1.1} in the sense of distributions.
\begin{theorem}[\cite{MS19,SZ20}]\label{posed dis}
Under the assumptions \eqref{eq1.2}-\eqref{eq1.5}, for any $\xi_{\tau}\in\mathcal{E}$ there exists a unique Shatah-Struwe solution
of problem \eqref{eq1.1} defined for all $t>\tau$. Moreover, this solution possesses the following dissipative estimate
\begin{equation}\label{eq2.2}
\|\xi_{u}(t)\|_{\mathcal{E}}+\|u\|_{L^{4}(t,t+1;L^{12}(\Omega))}\leq Q(\|\xi_{\tau}\|_{\mathcal{E}})e^{-\beta(t-\tau)}+Q(\|g\|_{L_{b}^{2}(\mathbb{R};L^{2}(\Omega))}),
\end{equation}
where the positive constant $\beta$ and the monotone increasing function $Q$ are independent of $t,\tau$, $\xi_{\tau}$ and $g$.
\end{theorem}
According to the general scheme of \cite{CV02} to construct the uniform attractor, we consider a family of equations with symbols $h\in\Sigma =\Cal H(g)$:
\begin{equation}\label{eq2.3}
\begin{cases}
\partial_{t}^{2}u+\gamma\partial_{t}u-\Delta u+f(u)=h(t),\\
u\big|_{\partial\Omega}=0,\ \ \ \xi_{u}\big|_{t=\tau}=\xi_{\tau}.
\end{cases}
\end{equation}
Problem \eqref{eq2.3} is well-posed for each $h\in\Sigma$ and then generate a family of dynamical processes $\{U_h(t,\tau),h\in\Sigma\}$ in the energy phase space $\mathcal E$ which satisfy the translation identity
\begin{equation}\label{trans}
 U_{h}(t+s,\tau+s)=U_{T(s)h}(t,\tau),\ \ t\ge\tau,\ \ \tau\in\mathbb R,\ \ s\ge0.
 \end{equation}
Moreover, estimate \eqref{eq2.2} implies that the family of processes $\{U_{h}(t,\tau),h\in\Sigma\}$ corresponding to \eqref{eq2.3} has a bounded uniformly absorbing set $\Cal B_{0}$ in $\mathcal{E}$.
Since $\mathcal E$ is a reflexive Banach space, the absorbing set $\Cal B_0$ is compact and metrizable in the weak topology of $\mathcal E$. 
It is straightforward to verify the maps
$(\xi,h)\to U_h(t,\tau)\xi$ are weakly continuous from $\mathcal E\times\Sigma$ to $\mathcal E$ for every fixed $t$ and $\tau$.
Then, the family of processes $\{U_{h}(t,\tau),h\in\Sigma\}$ possesses a weak uniform attractor $\Cal{A}_{\Sigma}^{w}$ which satisfies \eqref{1.rep}.
\begin{theorem}[\cite{MS19,SZ20}]\label{wunifexi}
Let the assumptions \eqref{eq1.2}-\eqref{eq1.5} hold, and in addition let the external force $g$ be space or time regular. Then there exists a uniform attractor $\Cal{A}_{\Sigma}$ for equation \eqref{eq1.1} in the strong topology of the energy space $\mathcal E$ and it coincides with the weak attractor $\Cal{A}_{\Sigma}^{w}$ constructed above.
\end{theorem}

Let us recall the definition of the Kolmogorov $\varepsilon$-entropy of a (pre)compact set $K$ in a metric space $X$. For a given $\varepsilon>0$, let $N_{\varepsilon}(K;X)=N_{\varepsilon}(K)$ be the minimal number of $\varepsilon$-balls in $X$ which is necessary to cover $K$. Since $K$ is (pre)compact, the number $N_{\varepsilon}(K)<\infty$ for all $\varepsilon>0$ due to Hausdorff's criteria.
\begin{definition}[\cite{CV02,KT59}]\label{def2.1}
The Kolmogorov $\varepsilon$-entropy of the set $K$ in the space $X$ is the number
\begin{equation*}
\mathbb{H}_{\varepsilon}(K;X)=\mathbb{H}_{\varepsilon}(K)=\log_{2}N_{\varepsilon}(K).
\end{equation*}
\end{definition}

\section{The weak topology entropy $\mathbb{H}_{\varepsilon}^w(A,X;\,B,Y)$}\label{s3.m}
Let $X$ and $Y$ be two Banach spaces with the embedding $X$ to $ Y^*$ is compact (w.r.t. the strong topology of $Y^*$) and $Y$ is reflexive.

In the following, for the bounded subsets $A\subset X$ and $B\subset Y$, we will define an entropy $\mathbb{H}_{\varepsilon}^w(A,X;\,B,Y)$ with respect to the weak topology of $Y$ for the subsets-spaces pair $(A,X;\,B, Y)$.

\begin{definition}
Let $A\subset X$ and $B\subset Y$ and $\varepsilon>0$ be given. An $(A,X;\,B,Y;\,\varepsilon)$-weak partition of $B$ is a partition $\cup_{i\in \Lambda}V_i$ which satisfies $B\subset \cup_{i\in \Lambda}V_i$ and
for each $V_i$ $(i\in \Lambda)$, the following estimate hold:
\[
|\langle f,\,v-\bar{v} \rangle|<\varepsilon ~~\forall~f\in A,~~v,\bar{v}\in B\cap V_i.
\]
\end{definition}

For any $\varepsilon>0$, let $N_{\varepsilon}^w(A,X;\,B,Y)$ be the minimal number of the cardinality of $\Lambda$ with $\cup_{i\in \Lambda}V_i$ is an $(A,X;\,B,Y;\,\varepsilon)$-weak partition of $B$.

 If $A$ is bounded in $X$ and $B$ is bounded in $Y$, then by the compactness of the embedding $X\hookrightarrow Y^*$ and reflexive property of $Y$, we have that the number $N_{\varepsilon}^w(A,X;\,B,Y)<\infty$ for all $\varepsilon>0$.

Indeed, we will construct a special weak partition of $B$ that each $V_{i}$, $i\in\Lambda$ can be presented as the form of a neighborhood in weak topology. So let us first recall that for any $v_0\in  Y$, given $\eta>0$ and a finite set $\{f_1,f_2,\cdots,f_k\}\subset Y^*$, the set
\begin{equation}\label{nerghborhood}
V_{v_0}=V(f_1,f_2,\cdots,f_k;\eta):=\{v\in Y:\,|\langle f_i,\,v-v_0 \rangle|<\eta,~\forall\,i=1,2,\cdots,k\}
\end{equation}
is an open neighborhood of $v_0$ for the weak topology of $Y$, and we can obtain a basis of neighborhoods (w.r.t. weak topology of $Y$) of $v_0$ by varying $\eta$, $k$ and the $f_i's$ in $Y^*$. Here $\langle\cdot,\,\cdot \rangle$ is the dual product of $Y^*$ and $Y$.

Since for any $\bar{\epsilon}>0$, $A$ has finite $\bar{\epsilon}$-net $\{f_{1},\cdots,f_{m}\}$ in $Y^*$, that is, for every $f\in A$
there exists $f_{i}$, $i\in \{1,\cdots,m\}$ such that
\begin{equation}\label{a1}
\begin{split}
&\|f-f_{i}\|_{Y^*}\leqslant \bar{\epsilon}.
\end{split}
\end{equation}

Then, for each $v_0\in B$ and $\eta>0$, we define
\begin{equation}\label{a2}
\mathcal{O}_{v_0}=\mathcal{O}_{v_0}(f_{1},\cdots,f_{m};\eta)=\{v\in Y;|\langle f_{i},v-v_0\rangle|<\eta,~i=1,\cdots,m\}.
\end{equation}
Clearly, $\mathcal{O}_{v_0}$ is a neighborhood of $v_0$ for the weak topology of $Y$. By the reflexive property of $Y$ we know that $B$ is weak compact, thus, there exist finite $v_{1},\cdots,v_{M}\in B$ such that
\begin{equation}\label{a3}
B\subset\cup_{j=1}^{M}\mathcal{O}_{v_{j}}.
\end{equation}
Note that, for any $f\in A$ and any $v,\bar{v}\in \mathcal{O}_{v_j}\cap B$, we have that
\begin{equation}\label{3.5}
\begin{split}
|\langle f,\,v-\bar{v} \rangle| &\leqslant |\langle f-f_i,\,v-\bar{v} \rangle| +|\langle f_i,\,v-\bar{v} \rangle|\\
 &\leqslant |\langle f-f_i,\,v-\bar{v} \rangle| +|\langle f_i,\,v-v_j \rangle|+|\langle f_i,\,v_j-\bar{v} \rangle|\\
 &\leqslant 2\|B\|_{Y}\bar{\epsilon}+2\eta,
\end{split}
\end{equation}
where $\|B\|_{Y}$ is the bound of $B$ in $Y$. Then, by taking first $\bar{\epsilon}$ small, e.g., $\bar{\epsilon}<\varepsilon /4\|B\|_{Y}$, to find $f_1,\cdots, f_m$, then taking $\eta$ small enough, e.g., $\eta <\varepsilon /4$, we can find that the special partition defined in \eqref{a3} is an $(A,X;\,B,Y;\,\varepsilon)$-weak partition of $B$ with finite neighborhoods.

\begin{definition}
Let $A\subset X$ and $B\subset Y$ and $\varepsilon>0$ be given. The $(A,X;\,B,Y;\,\varepsilon)$-weak topology entropy $\mathbb{H}_{\varepsilon}^w(A,X;\,B,Y)$ of the set $B$ in $Y$ corresponding to set $A$ in $X$  is the number
\begin{equation*}
\mathbb{H}_{\varepsilon}^w(A,X;\,B,Y)=\log_{2}N_{\varepsilon}^w(A,X;\,B,Y).
\end{equation*}
\end{definition}
\begin{remark}
The definitions proposed above still valid if only know $A\subset Y^{\ast}$ is compact. Here, we  involve the additional space $X$ is just designed for our application in Sections 4 and 5, and we will estimate this bound explicitly and to make clear its dependencies for system \eqref{eq1.1}. For example, if $A$ and $B$ are some balls in $X$ and $Y$ respectively, e.g.,  for given $r,s\geqslant 0$, denote
\[
A_r:=\{u\in X:\, \|u\|_{X}\leqslant r\}\quad \text{and}\quad B_s:=\{v\in Y:\, \|v\|_{Y}\leqslant s\},
\]
then the $(A_r,X;\,B_s,Y;\,\varepsilon)$-weak topology entropy $\mathbb{H}_{\varepsilon}^w(A_r,X;\,B_s,Y)$ will depends only on the spaces $X$ and $Y$ and the size $r,s$. For this situation, we will denote $\mathbb{H}_{\varepsilon}^w(A_r,X;\,B_s,Y)$ by $\mathbb{H}_{\varepsilon}^w(r,X;\,s,Y)$.
\end{remark}

\section{Smoothness of uniform attractors}\label{s4}
\begin{proposition}\label{compattr}
Let the assumptions of Theorem \ref{wunifexi} hold, and in addition, let the external force $g$ satisfies \eqref{gt-reg} or \eqref{gs-reg}. Then, if $R$ is large enough, the closed ball $\Cal B_1$ of radius $R$ in the space $\Cal E^{1}:=[H^{2}(\Omega)\cap H_0^{1}(\Omega)]\times H_0^{1}(\Omega)$ 
is a uniformly attracting set for the process $U_g(t,\tau)$ associated with equation \eqref{eq1.1}. Namely, there
are positive constant $\beta$ and a monotone increasing function $Q$ such that, for any bounded set $B\subset \Cal E$,
$$
\dist_{\Cal E}(U_g(t+\tau,\tau)B,\Cal B_1)\le Q(\|B\|_{\Cal E})e^{-\beta t}, \ \ t\geq0
$$
holds uniformly with respect to $\tau\in\R$.
\end{proposition}
\begin{proof}
We split the solution $u(t)$ as $u(t)=v(t)+w(t)$, where $v(t)$ solves the linear problem
\begin{equation}\label{eqv}
\Dt^{2}v+\gamma\Dt v-\Delta v=0, \ \ \xi_v\big|_{t=\tau}=\xi_\tau,
\end{equation}
the reminder $w$ satisfies
\begin{equation}\label{eqw}
\Dt^{2}w+\gamma\Dt w-\Delta w+f(u)=g, \ \ \xi_w\big|_{t=\tau}=0.
\end{equation}
Multiplying \eqref{eqv} by $\Dt v+\beta v$ with sufficiently small $\beta>0$,
we obtain that the solution $v(t)$ satisfies the following estimate
\begin{equation}\label{vesti}
\|\xi_v(t)\|_{\Cal E}\le Q(\|\xi_\tau\|_{\Cal E})e^{-\beta(t-\tau)}, \ \ t\geq\tau,
\end{equation}
where the positive constant $\beta$ and the monotone increasing function $Q$ are independent of $t$, $\tau$ and $\xi_\tau$.

Now, we turn to the equation \eqref{eqw}. In the case of \eqref{gs-reg},
using the standard $\Cal E^{\alpha}$-energy estimate \eqref{linenergyest}, we have
$$
\|\xi_w(t)\|_{\Cal E^{\alpha}}\le C\int_\tau^{t}e^{-\beta(t-s)}(\|f(u)\|_{H^{\alpha}}+\|g(s)\|_{H^{\alpha}})ds\le C(\|f(u)\|_{L_b^{1}(\R;H^{\alpha})}+\|g\|_{L_b^{2}(\R;H^{\alpha})}).
$$
Using the H\"{o}lder inequality, the growth condition of $f$ and estimate \eqref{eq2.2}, we have
\begin{equation}\label{fest}
\begin{split}
&\|f(u)\|_{L^{1}(t,t+1;W^{1,\kappa})}\le C\|f'(u)\nabla u\|_{L^{1}(t,t+1;L^{\kappa})}\\
&\ \ \ \ \ \ \ \ \ \ \ \ \ \ \ \ \ \ \ \ \ \ \ \ \ \le C(1+\|u\|_{L^{4}(t,t+1;L^{12})}^{4})\|\nabla u\|_{L^{\infty}(t,t+1;L^{2})}\\
&\ \ \ \ \ \ \ \ \ \ \ \ \ \ \ \ \ \ \ \ \ \ \ \ \ \le Q(\|\xi_\tau\|_{\Cal E})e^{-\beta(t-\tau)}+Q(\|g\|_{L_b^{2}(\R;L^{2})}), \ \ t\geq\tau,
\end{split}
\end{equation}
where $\frac{1}{\kappa}=\frac{1}{2}+\frac{p-1}{12}$. Using the embedding $$W^{1,\kappa}\subset H^{\alpha} \ \ \text{with} \ \  \frac{1}{2}=\frac{1}{\kappa}-\frac{1-\alpha}{3},$$
i.e. $$\alpha=\frac{5-p}{4}>0,$$ we obtain that
$$
\|f(u)\|_{L^{1}(t,t+1;H^{\alpha})}\le Q(\|\xi_\tau\|_{\Cal E})e^{-\beta(t-\tau)}+Q(\|g\|_{L_b^{2}(\R;L^{2})}).
$$
Finally, we arrive at the estimate
\begin{equation}\label{westi}
\|\xi_w(t)\|_{\Cal E^{\alpha}}\le Q(\|\xi_\tau\|_{\Cal E})e^{-\beta(t-\tau)}+Q(\|g\|_{L_b^{2}(\R;H^{\alpha})}),
\end{equation}
where the constant $\beta$ and the monotone function $Q$ are independent of $t$, $\tau$, $g$ and $\xi_\tau$.

Now let \eqref{gt-reg} be satisfied. We split the solution $w(t)$ of equation \eqref{eqw} as $w(t)=y(t)+z(t)$, where $z(t)$ solves the linear wave equation
\begin{equation}\label{eqz}
\Dt^{2}z+\gamma\Dt z-\Delta z=g, \ \ \xi_z\big|_{t=\tau}=0
\end{equation}
and the function $y(t)$ solves
\begin{equation}\label{eqy}
\Dt^{2}y+\gamma\Dt y-\Delta y+f(u)=0, \ \ \xi_y\big|_{t=\tau}=0.
\end{equation}
Similar to the proof of estimate \eqref{westi}, we have
\begin{equation}\label{yesti}
\|\xi_y(t)\|_{\Cal E^{\alpha}}\le Q(\|\xi_\tau\|_{\Cal E})e^{-\beta(t-\tau)}+Q(\|g\|_{L_b^{2}(\R;L^{2})}).
\end{equation}
Then by differentiating \eqref{eqz} with respect to time and writing $\bar{z}:=\Dt z$, we get that
$$
\Dt^{2}\bar z+\gamma\Dt \bar z-\Delta \bar z=g'(t), \ \ \xi_{\bar z}\big|_{t=\tau}=(0,g(\tau)).
$$
Using the standard energy estimate and together with the Sobolev embedding $H_{b}^{1}(\mathbb{R};L^{2}(\Omega))\subset L^{\infty}_{b}(\mathbb{R};L^{2}(\Omega))$, we have
\begin{equation}\label{zesti1}
\|\xi_{\bar{z}}(t)\|_{\Cal E}\le C\|g\|_{H_b^{1}(\R;L^{2})}, \ \ t\geq\tau.
\end{equation}
Expressing the term $\Delta z(t)$ from equation \eqref{eqz} and taking the $L^{2}$-norm from the both sides of the equation obtained, and combining with \eqref{zesti1} we have that, for all $t\geq\tau$
\begin{equation}\label{zesti2}
\|z(t)\|_{H^{2}} \le C\|g\|_{H_{b}^{1}(\mathbb{R};L^{2}(\Omega))}.
\end{equation}
Combining \eqref{zesti1} and \eqref{zesti2}, we obtain that
\begin{equation}\label{zesti}
\|\xi_{z}(t)\|_{\mathcal{E}^{1}}\leq C\|g\|_{H_{b}^{1}(\mathbb{R};L^{2}(\Omega))}.
\end{equation}

From \eqref{westi}, \eqref{yesti} and \eqref{zesti}, 
we conclude that
\begin{equation}\label{westimate}
\|\xi_w(t)\|_{\Cal E^{\alpha}}\le Q(\|\xi_\tau\|_{\Cal E})e^{-\beta(t-\tau)}+Q(\|g\|_{W}),
\end{equation}
where the symbol $W$ means the space $H_{b}^{1}(\mathbb{R};L^{2}(\Omega))$ if \eqref{gt-reg} is satisfied or $L_{b}^{2}(\mathbb{R};H^{1}(\Omega))$ if \eqref{gs-reg} is satisfied, the positive constant $\beta$ and the monotone increasing function $Q$ are independent of $t,\tau$, $g$ and $u$.
According to the estimates \eqref{vesti} and \eqref{westimate}, the set
$$
\Cal B_\alpha:=\{\xi\in\Cal E^{\alpha};\|\xi\|_{\Cal E^{\alpha}}\le R\}
$$
is a compact uniformly (w.r.t. $\tau\in\R$) attracting set for the process $U_g(t,\tau)$ in $\Cal E$ if $R$ is large enough.

In order to get the higher regularity, we will use standard bootstrap arguments. If we take $\xi_\tau\in \Cal B_\alpha$ from the very beginning, we have
$$
\|\xi_v(t)\|_{\Cal E^{\alpha}}\le Q(\|\xi_\tau\|_{\Cal E^{\alpha}})e^{-\beta(t-\tau)}, \ \ t\geq\tau,
$$
which together with \eqref{westimate} shows that the dynamical process $U_g(t,\tau)$ is well
defined and dissipative in the higher energy space $\Cal E^{\alpha}$ as well. Due to that $H^{\alpha}\subset L^{\frac{6}{3-2\alpha}}$, we can improve
estimate \eqref{fest} as follows:
\begin{equation*}
\begin{split}
&\|f(u)\|_{L^{1}(t,t+1;W^{1,\kappa_1})}\le C(1+\|u\|_{L^{4}(t,t+1;L^{12})}^{4})\|\nabla u\|_{L^{\infty}(t,t+1;L^{\frac{6}{3-2\alpha}})}\\
& \ \ \ \ \ \ \ \ \ \ \ \ \ \ \ \ \ \ \ \ \ \ \ \ \ \ \le Q(\|\xi_\tau\|_{\Cal E^{\alpha}})e^{-\beta(t-\tau)}+Q(\|g\|_{W}), \ \ t\geq\tau,
\end{split}
\end{equation*}
where $\frac{1}{\kappa_1}=\frac{3-2\alpha}{6}+\frac{p-1}{12}$. Using the embedding $W^{1,\kappa_1}\subset H^{\alpha_1}$ with  $\alpha_1=\frac{5-p}{4}+\alpha$, we have
$$
\|f(u)\|_{L^{1}(t,t+1;H^{\alpha_1})}\le Q(\|\xi_\tau\|_{\Cal E^{\alpha}})e^{-\beta(t-\tau)}+Q(\|g\|_{W}).
$$
Thus, we get
\begin{equation}\label{westi1}
\|\xi_w(t)\|_{\Cal E^{\alpha_1}}\le Q(\|\xi_\tau\|_{\Cal E^{\alpha}})e^{-\beta(t-\tau)}+Q(\|g\|_{W}).
\end{equation}
By the transitivity of an exponential attraction established in \cite{FGMZ04}, the dynamical process $U_g(t,\tau)$ on $\Cal B_\alpha$
has an exponentially attracting set $\Cal B_{\alpha_1}\subset \Cal E^{\alpha_1}$. Then iterating the above procedure in finitely many steps, we get the exponentially attracting
ball $\Cal B_1$ in the space $\Cal E^{1}$.
\end{proof}

\begin{corollary}\label{compabsorb}
Let the assumptions of Proposition \ref{compattr} hold, and in addition let the initial data $\xi_\tau\in \Cal E^{1}$.
Then the solution $u$ of problem \eqref{eq1.1} satisfies $\xi_u(t)\in \Cal E^{1}$ for all $t\geq\tau$ and the following estimate holds:
\begin{equation}\label{reg}
\|\xi_u(t)\|_{\Cal E^{1}}\le Q(\|\xi_\tau\|_{\Cal E^{1}})e^{-\beta(t-\tau)}+Q(\|g\|_W),
\end{equation}
where the symbol $W$ means the space $H_{b}^{1}(\mathbb{R};L^{2}(\Omega))$ if \eqref{gt-reg} is satisfied or $L_{b}^{2}(\mathbb{R};H^{1}(\Omega))$ if \eqref{gs-reg} is satisfied, the positive constant $\beta$ and the monotone increasing function $Q$ are independent of $t,\tau$, $g$ and $u$.
\end{corollary}
The proof of estimate \eqref{reg} can be verified analogously to the proof of the estimation of 
$w$-component in the higher energy space in Proposition 
\ref{compattr}. 

Then, we will prove that $\Cal B_1$ is actually a compact uniformly (w.r.t. $h\in\Sigma$) attracting set
in $\Cal E$ for the family of processes $\{U_h(t,\tau),h\in\Sigma\}$.
\begin{lemma}
Let the family of processes $\{U_h(t,\tau),h\in\Sigma\}$ satisfies translation identity \eqref{trans}
and be weakly continuous. Then any compact uniformly (w.r.t. $\tau\in\R$) attracting set for the
process $U_g(t,\tau)$ is a simultaneously compact uniformly (w.r.t. $h\in\Sigma$) attracting set for the
family $\{U_h(t,\tau),h\in\Sigma\}$.
\end{lemma}
\begin{proof}
Due to the translation identity \eqref{trans}, any compact uniformly (w.r.t. $\tau\in\R$) attracting set for the process $U_g(t,\tau)$
is a simultaneously compact uniformly (w.r.t. $h\in \Sigma_0$) attracting set for the family $\{U_h(t,\tau),h\in\Sigma_0\}$. 

If $P$ is a compact uniformly (w.r.t. $h\in \Sigma_0$) attracting set for the family $\{U_h(t,\tau),h\in\Sigma_0\}$, that is, for every $\eb>0$ 
and every bounded set $B\subset \Cal E$, there exists $T_0=T_0(B,\eb)>0$ such that
\begin{equation}\label{attract shift}
\cup_{h\in\Sigma_0}U_h(t,0)B\subset \Cal O_{\eb/2}(P), \ \ \forall t\geq T_0. 
\end{equation}
We will prove that
\begin{equation}\label{attract hull}
\cup_{h\in\Sigma}U_h(t,0)B\subset\Cal O_{\eb}(P), \ \ \forall t\geq T_0. 
\end{equation}
For any $y\in \cup_{t\geq T_0}\cup_{h\in\Sigma}U_h(t,0)B$, there exist $t\geq T_0$, $h\in\Sigma$ and $x\in B$ such that $y=U_h(t,0)x$. For such $h\in\Sigma$, there exists $\{h_n\}\in\Sigma_0$ such that
$h_n\rightarrow h$ in the weak topology of $L_{loc}^{2}(\R;H)$ as $n\rightarrow\infty$. Then the weak continuity of processes implies that
\begin{equation}\label{weak conv}
U_{h_n}(t,0)x\rightharpoonup U_h(t,0)x \ \ \text{in} \ \ \Cal E \ \ \text{as} \ \ n\rightarrow\infty.
\end{equation}
For every $\eta>0$, we have $P\subset\cup_{x\in P}B(x,\eta)$.
Since $P$ is a compact set in $\Cal E$, we have a finite subcovering such that $P\subset\cup_{i=1}^{N}B(x_i,\eta)$, $x_i\in P$. Then, $\Cal O_{\eb/2}(P)\subset\cup_{i=1}^{N}B(x_i,\eta+\eb/2)$. Due to that
$\{U_{h_n}(t,0)x\}\subset\Cal O_{\eb/2}(P)$ if $t\geq T_0$, there exists a subsequence of $\{U_{h_n}(t,0)x\}$ belongs to some ball $B(x_i,\eta+\eb/2)$. By Mazur theorem, we know that its weak limit $$U_h(t,0)x=y\in [B(x_i,\eta+\eb/2)]_{\Cal E}\subset[\Cal O_{\eta+\eb/2}(P)]_{\Cal E}.$$
 Let $\eta\rightarrow0$, we have that $y\in [\Cal O_{\eb/2}(P)]_{\Cal E}\subset\Cal O_{\eb}(P)$. This means that \eqref{attract hull} holds, thus $P$ is a compact uniformly (w.r.t. $h\in\Sigma$) attracting set.
On the other hand, \eqref{attract hull} always implies \eqref{attract shift} since $\Sigma_0\subset\Sigma$.
\end{proof}

Then by the minimality of the uniform attractor among closed
uniformly attracting sets we know that $\Cal A_\Sigma\subset\Cal B_1$ and is bounded in the space $\Cal E^{1}$. Moreover, by Corollary \ref{compabsorb} we have
\begin{equation}\label{embed}
\cup_{h\in\Sigma}U_h(t,0)\Cal A_\Sigma\subset\cup_{h\in\Sigma}U_h(t,0)\Cal B_1\subset\Cal B_1, \ \ t\geq T_0.
\end{equation}

\section{Uniform exponential attracting sets}\label{s.5}
In this section, we will construct a family of sets
which enjoys uniform exponential attracting property and has uniform finite fractal dimension.
\begin{lemma}\label{lem4.1}
Let $X_{0}$ and $X_{1}$ be two Banach spaces such that $X_{1}$ is compactly embedded into $X_{0}$ and $\mathbb{B}$ be a bounded subset of $X_{0}$. For any fixed $h\in\Sigma$,
let $U_{h}(n):=U_{h}(n+1,n),~n\in\mathbb{Z}$ be a family of discrete dynamical processes on $X_{0}$ satisfying the following properties:
\begin{itemize}
\item[$(a)$] for every $h\in\Sigma$, the process $U_{h}(n):\mathbb{B}\rightarrow\mathbb{B}$;
\item[$(b)$] the process $U_{h}(n)$ admits a decomposition of the form
\begin{equation*}
U_{h}(n)=\mathcal{C}_{h}(n)+\mathcal{L}_{h}(n),
\end{equation*}
where $\mathcal{C}_{h}(n):\mathbb{B}\rightarrow X_{0}$ satisfies
\begin{equation}\label{eq5.1}
\|\mathcal{L}_{h}(n)\xi_{1}-\mathcal{L}_{h}(n)\xi_{2}\|_{X_{0}}\leq\kappa\|\xi_{1}-\xi_{2}\|_{X_{0}},\ \ \ \forall~\xi_{1},~\xi_{2}\in\mathbb{B}
\end{equation}
for some $\kappa<\frac{1}{2}$ and $\mathcal{C}_{h}(n):\mathbb{B}\rightarrow X_{1}$ satisfies
\begin{equation}\label{eq5.2}
\|\mathcal{C}_{h}(n)\xi_{1}-\mathcal{C}_{h}(n)\xi_{2}\|_{X_{1}}\leq L\|\xi_{1}-\xi_{2}\|_{X_{0}},\ \ \ \forall~\xi_{1},~\xi_{2}\in\mathbb{B}
\end{equation}
for some $L>0$.
\end{itemize}
Then, there exists a family of time-dependent sets $\mathcal{M}_{h}(n)$, $n\in\mathbb{Z}$, which satisfy
\begin{itemize}
\item[$(1)$] $\mathcal{M}_{h}(n)\subset\mathbb{B}$ for every $n\in\mathbb{Z}$ and their fractal dimension is uniformly bounded
\begin{equation}\label{eq5.3}
dim_{F}(\mathcal{M}_{h}(n);X_{0})\leq C_{1},
\end{equation}
where the constant $C_{1}$ is independent of $n\in\mathbb{Z}$;
\item[$(2)$] the uniform exponential attraction property:
\begin{equation}\label{eq5.4}
dist_{X_{0}}(U_{h}(n+k,n)\mathbb{B},\mathcal{M}_{h}(n+k))\leq C_{2}e^{-\sigma k},
\end{equation}
where the positive constants $C_{2}$ and $\sigma$ are independent of $n\in\mathbb{Z}$ and $k\in\mathbb{N}$.
\end{itemize}
\end{lemma}

\begin{proof}
The proof is similar to \cite{EMZ00,EMZ05}. For the convenience of the reader, we present the details
below.
Since $\mathbb{B}$ is bounded in $X_{0}$, there exists a ball $B(x_{0},R;X_{0})$ of radius $R$ centered at $x_{0}\in\mathbb{B}$ in $X_{0}$ such that $\mathbb{B}\subset B(x_{0},R;X_{0})$. We set $V_{0}:=\{x_{0}\}\subset\mathbb{B}$. We also fix an arbitrary dynamical process $U_{h}$ satisfying the above assumptions~$(a)$ and $(b)$.

Now we construct a family of the sets $V_{k}(n)\subset\mathbb{B}$ by induction with respect to $k$ such that $V_{k}(n)$ is an $R_{k}:=R(\kappa+\frac{1}{2})^{k}$-net of $U_{h}(n,n-k)\mathbb{B}$. Assume that the required sets are already been constructed for some $k=l$, then we construct the next set $V_{l+1}(n+1)$ preserving these properties.

It follows from \eqref{eq5.2} that
\begin{equation*}
\mathcal{C}_{h}(n)(U_{h}(n,n-l)\mathbb{B})\subset \cup_{x\in V_{l}(n)}B(\mathcal{C}_{h}(n)x,LR_{l};X_{1}).
\end{equation*}
Since the embedding $X_{1}\subset X_{0}$ is compact, we cover this ball by a finite number of $\frac{1-2\kappa}{4}R_l$-balls in $X_0$ with centers $y_i$ and the minimal number of balls in this covering can be estimated as follows:
$$
N_{\frac{1-2\kappa}{4}R_l}(B(\mathcal{C}_{h}(n)x,LR_{l};X_{1});X_0)
=N_{\frac{1-2\kappa}{4L}}(B(0,1;X_{1});X_0):=N.
$$
Crucial for us that the number $N$ is independent of $l$, $R$, $\kappa$ and $L$. It follows from \eqref{eq5.1} that the family of balls with centers $\mathcal{L}_{h}(n)x$, $x\in V_l(n)$ and with
radius $\kappa R_l$ covers $\mathcal{L}_{h}(n)(U_{h}(n,n-l)\mathbb{B})$. 
Consequently, we conclude that
\begin{equation*}
\begin{split}
U_{h}(n+1,n-l)\mathbb{B}&=U_{h}(n)U_{h}(n,n-l)\mathbb{B}\\
&\subset\cup_{x\in V_{l}(n)}B(y_i+\mathcal{L}_{h}(n)x,\frac{1-2\kappa}{4} R_{l}+\kappa R_{l};X_{0})\\
& =\cup_{x\in V_{l}(n)}B(y_i+\mathcal{L}_{h}(n)x,\frac{1+2\kappa}{4} R_{l};X_{0}),
\end{split}
\end{equation*}
and the number of balls in this system is not greater than $N\cdot\sharp V_{l}(n)$.

Increasing the radius of every ball in this covering by a factor of two, we can assume that the centres of this covering belong to $U_{h}(n+1,n-l)\mathbb{B}$. We denote by $V_{l+1}(n+1)$ the new centres of this covering and condition $(a)$ for $U_{h}(n)$ guarantees that $V_{l+1}(n+1)\subset\mathbb{B}$, we also note that $R_{l+1}=\frac{1+2\kappa}{2}R_{l}=R(\kappa+\frac{1}{2})^{l+1}$. Thus, the required sets $V_{k}(n)$ are constructed for every $n\in\mathbb{Z}$ and $k\in\mathbb{N}$.

According to the above construction, we have
\begin{equation}\label{eq5.5}
\begin{split}
&\sharp V_{k}(n)\leq N^{k},\ \ k\in\mathbb{N},\ n\in\mathbb{Z},\\
&\mathrm{dist}_{X_{0}}(U_{h}(n,n-k)\mathbb{B},V_{k}(n))\leq R(\kappa+\frac{1}{2})^{k}.
\end{split}
\end{equation}
We define the sets $E_{k}(n)=E_{k}^{h}(n)$ by
\begin{equation}\label{eq5.6}
E_{1}(n):=V_{1}(n),\ \ E_{k+1}(n+1):=V_{k+1}(n+1)\cup U_{h}(n)E_{k}(n), \ \ k\in\mathbb{N},\ n\in\mathbb{Z}.
\end{equation}
From \eqref{eq5.5} and \eqref{eq5.6}, we have
\begin{equation}\label{eq5.7}
\begin{split}
&\sharp E_{k}(n)\leq kN^{k},\\
&\mathrm{dist}_{X_{0}}(U_{h}(n,n-k)\mathbb{B},E_{k}(n))\leq R(\kappa+\frac{1}{2})^{k}.
\end{split}
\end{equation}
The required family of sets $\mathcal{M}_{h}(n)$ can be defined by
\begin{equation}\label{eq5.8}
\mathcal{M}_{h}(n):=\cup_{k=1}^{\infty}E_{k}(n), \ \ n\in\mathbb{Z}.
\end{equation}

Let us verify that $\mathcal{M}_{h}(n)$ defined by \eqref{eq5.8} satisfies \eqref{eq5.3} and \eqref{eq5.4}.

Indeed, since $E_{k}(n)\subset\mathcal{M}_{h}(n)$ and $\kappa+\frac{1}{2}<1$, the uniform exponential attraction property follows from $\eqref{eq5.7}$:
\begin{equation*}
\mathrm{dist}_{X_{0}}(U_{h}(n,n-k)\mathbb{B},\mathcal{M}_{h}(n))\leq R(\kappa+\frac{1}{2})^{k}=Re^{-k\ln\frac{1}{\kappa+\frac{1}{2}}}.
\end{equation*}

Now, there remains to verify the finite dimensionality of $\mathcal{M}_{h}(n)$.
We fix $\varepsilon>0$ and choose the smallest integer $k=k(\varepsilon)$ such that $R(\kappa+\frac{1}{2})^{k}\leq\varepsilon$. Since, by definition \eqref{eq5.6},
\begin{equation*}
E_{k}(n)=\cup_{l=0}^{k-1}U_{h}(n,n-l)V_{k-l}(n-l)
\end{equation*}
and from the construction of $V_{k}(n)$, we have $V_{k-l}(n-l)\subset U_{h}(n-l,n-k)\mathbb{B}$. Then, by \eqref{eq5.5} we obtain
\begin{equation*}
\begin{split}
\cup_{k\geq k(\varepsilon)}E_{k}(n)&\subset\cup_{k\geq k(\varepsilon)}\cup_{l=0}^{k-1}U_{h}(n,n-l)U_{h}(n-l,n-k)\mathbb{B}\\
& =\cup_{k\geq k(\varepsilon)}U_{h}(n,n-k)\mathbb{B}\\
&\subset\cup_{v\in V_{k}(n)}B(v,\varepsilon;X_{0}).
\end{split}
\end{equation*}
Thus,
\begin{equation*}
\begin{split}
N_{\varepsilon}(\mathcal{M}_{h}(n);X_{0})&\leq
N_{\varepsilon}(\cup_{k\le k(\varepsilon)}E_{k}(n))+N_{\varepsilon}(\cup_{k> k(\varepsilon)}E_{k}(n))\\
&\leq\sum_{k\le k(\varepsilon)}\sharp E_{k}(n)+\sharp V_{k(\eb)+1}(n)\\
& \leq(k(\varepsilon)+1)^{2} N^{k(\varepsilon)+1}.
\end{split}
\end{equation*}
Consequently,
\begin{equation*}
\mathrm{dim}_{F}(\mathcal{M}_{h}(n);X_{0}):=\limsup\limits_{\varepsilon\rightarrow0^{+}} \frac{\log_{2}N_{\varepsilon}(\mathcal{M}_{h}(n);X_{0})}{\log_{2}\frac{1}{\varepsilon}}\leq(\log_{2}\frac{1}{\kappa+\frac{1}{2}})^{-1}\log_{2}N.
\end{equation*}
\end{proof}

Then, we will apply Lemma \ref{lem4.1} to problem \eqref{eq1.1}.
\begin{lemma}\label{lem5.2}
Let assumptions \eqref{eq1.2}-\eqref{eq1.5} hold. Then, there exists some $T>0$ such that, for any arbitrary fixed $h\in\Sigma$ and $\tau\in\mathbb{R}$, the family of discrete dynamical processes $U_{h}^{\tau}(m,n):=U_{h}(\tau+mT,\tau+nT)$, $m,n\in\mathbb{Z},m\geq n$ possesses a family sets $\mathcal{M}_{h}^{\tau}(n)$, $n\in\mathbb{Z}$ which satisfies
\begin{equation}\label{eq5.9}
dist_{\mathcal{E}}(U_{h}^{\tau}(m,n)\Cal B_1,\mathcal{M}_{h}^{\tau}(m))\leq \nu e^{-\sigma(m-n)T},
\end{equation}
where the positive constants $\nu$ and $\sigma$ are independent of $m,n\in\mathbb{Z}$, $m\geq n$ and $\tau\in\mathbb{R}$. And for every $\delta>0$, there exists $\varepsilon_{0}>0$ such that
\begin{equation}\label{eq5.10}
N_{\varepsilon}(\mathcal{M}_{h}^{\tau}(n);\mathcal{E})\leq (\frac{1}{\varepsilon})^{\mathcal{N}+\delta}, \ \ \text{for all}~0<\varepsilon\leq\varepsilon_{0},
\end{equation}
where the positive constant $\mathcal{N}$ is independent of $n\in\mathbb{Z}$ and $\tau\in\mathbb{R}$.
\end{lemma}
\begin{proof}
Thanks to the translation identity \eqref{trans}, we know that for any fixed $\tau\in\R$ and $h\in\Sigma$, we can find $h'\in\Sigma$ such that
$$
U_{h'}(t+\tau,\tau)x=U_h(t,0)x, \ \ \forall~t\geq0, \ x\in \Cal E.
$$
Therefore, we set $\tau=0$ for simplicity.

Let $u_1(t)$ and $u_2(t)$ be two solutions of \eqref{eq1.1} with different initial data $\xi_0^{1},\xi_0^{2}\in\Cal B_1$ starting at $t=0$ and with the same external force $h\in\Sigma$. Let $\theta(t)=u_1(t)-u_2(t)$, then this function solves
\begin{equation}\label{dif}
\Dt^2\theta+\gamma\Dt\theta-\Dx\theta+l(t)\theta=0, \ \xi_{\theta}\big|_{t=0}=\xi_{0}^{1}-\xi_{0}^{2},
\end{equation}
where $l(t):=\int_0^1f'(su_1(t)+(1-s)u_2(t))\,ds$.
We split the solution $\theta(t)=v(t)+w(t)$, where $v(t)$ satisfies the equation
\begin{equation}\label{eq3.2}
\partial_{t}^{2}v+\gamma\partial_{t}v-\Delta v=0,\ \  v\big|_{\partial\Omega}=0,\ \ \xi_{v}\big|_{t=0}=\xi_0^{1}-\xi_0^{2}
\end{equation}
and the function $w(t)$ satisfies
\begin{equation}\label{eq3.3}
\partial_{t}^{2}w+\gamma\partial_{t}w-\Delta w+l(t)\theta=0,\ \ w\big|_{\partial\Omega}=0, \ \ \xi_{w}\big|_{t=0}=0.
\end{equation}

For the linear equation \eqref{eq3.2}, we get
\begin{equation}\label{eq3.4}
\|\xi_{v}(t)\|_{\mathcal{E}}\leq C\|\xi_{0}^{1}-\xi_{0}^{2}\|_{\mathcal{E}} e^{-\beta t},
\end{equation}
where the positive constants $C$ and $\beta$ are independent of $t$, $g$ and $\xi_0^{1},\xi_{0}^{2}$. 
From Corollary \ref{compabsorb}, we know that
$$
\|\xi_{u_i}(t)\|_{\Cal E^{1}}\le C, \ \ \forall t\geq\tau, \ i=1,2,
$$
where the constant $C$ is independent of $t$, $u_i$ and $h$. Using the embedding $H^{2}\subset C$, we find that
\begin{equation}\label{difest}
\|\xi_w(t)\|_{\Cal E^{1}}\le C\int_0^{t}e^{-\beta(t-s)}\|l(t)\theta\|_{H_0^{1}}ds\le C\|\nabla \theta\|_{L^{\infty}(0,t;L^{2})}\le C\|\xi_\theta\|_{L^{\infty}(0,t;\Cal E)}.
\end{equation}
Next, we estimate the term $\|\xi_{\theta}\|_{L^{\infty}(0,t;\mathcal{E})}$. Multiplying \eqref{dif} by $\partial_{t}\theta$ and integrating over $\Omega$, we end up with
\begin{equation*}
\frac{d}{dt}\|\xi_{\theta}(t)\|_{\mathcal{E}}^{2}+2\gamma\|\partial_{t}\theta\|^{2}=2(l(t)\theta,\partial_{t}\theta).
\end{equation*}
Using the growth condition \eqref{eq1.2}, we get
\begin{equation*}\label{*1}
\begin{split}
2|(l(t)\theta,\partial_{t}\theta)|&\leq C((1+|u_{1}|^{p-1}+|u_{2}|^{p-1})|\theta|, |\partial_{t}\theta|)\\
& \leq C\|\partial_{t}\theta\|_{L^{2}}\|\theta\|_{L^{6}}(1+\|u_{1}\|_{L^{12}}^{4}+\|u_{2}\|_{L^{12}}^{4})\\
& \leq C\|\xi_{\theta}\|_{\mathcal{E}}^{2}(1+\|u_{1}\|_{L^{12}}^{4}+\|u_{2}\|_{L^{12}}^{4}),
\end{split}
\end{equation*}
By applying the Gronwall inequality, we obtain
$$
\|\xi_{\theta}(t)\|_{\mathcal{E}}\leq Ce^{Kt}\|\xi_0^{1}-\xi_0^{2}\|_{\mathcal{E}}.
$$
Inserting this into estimate \eqref{difest}, we have
$$
\|\xi_w(t)\|_{\Cal E^{1}}\le Ce^{Kt}\|\xi_0^{1}-\xi_0^{2}\|_{\mathcal{E}},
$$
where the positive constants $C$ and $K$ are independent of $t$, $h$, $\xi_0^{1}$ and $\xi_0^{2}$.

Taking $T_1>0$ such that $T_1=\frac{1}{\beta}\ln(4C)$ and fixing
\begin{equation}\label{T}
T:=\max\{T_0,T_1\},
\end{equation}
then the relation
\begin{equation}\label{embed1}
\cup_{h\in\Sigma}U_h(t,0)\Cal A_\Sigma\subset\cup_{h\in\Sigma}U_h(t,0)\Cal B_1\subset\Cal B_1, \ \ t\geq T
\end{equation}
together with estimates
\begin{equation*}\label{vdecay}
\|\xi_{v}(T)\|_{\mathcal{E}}\leq \frac{1}{4}\|\xi_{0}^{1}-\xi_{0}^{2}\|_{\mathcal{E}}
\end{equation*}
and
\begin{equation*}\label{wsmooth}
\|\xi_w(T)\|_{\Cal E^{1}}\le Ce^{KT}\|\xi_0^{1}-\xi_0^{2}\|_{\mathcal{E}}
\end{equation*}
implies that $U_{h}(T,0)$ satisfies the assumptions of Lemma \ref{lem4.1} with $\kappa=\frac{1}{4}$ and $L=Ce^{KT}$.
Thus, we can apply Lemma \ref{lem4.1} to the family of discrete dynamical processes $U_{h}^{0}(m,n)$, $m,n\in\mathbb{Z}$, $m\geq n$, which implies that these processes possess a family of sets $\mathcal{M}_{h}^{0}(n)$, $n\in\mathbb{Z}$ satisfies \eqref{eq5.9} and \eqref{eq5.10} with $\tau=0$.
\end{proof}
\begin{remark}
We have the following translation invariance:
\begin{equation}\label{eq5.18}
\mathcal{M}_{T(s)h}^{0}(l)=\mathcal{M}_{h}^{s}(l)
\end{equation}
for all $l\in\mathbb{Z}$ and $s\in\mathbb{R}$. 
Thus, the estimates \eqref{eq5.9} and \eqref{eq5.10} are independent of the specific choice of $h\in\Sigma_{0}$.
\end{remark}

\section{Estimates of the $\varepsilon$-entropy}
First, we will obtain a weak partition of $\Sigma_{iT,(i+1)T}$, $i\in\N^{+}$.

let $u_1(t)$ and $u_2(t)$ be two solutions of \eqref{eq2.3}
staring from the same initial data $\xi\in\Cal A_\Sigma$ at $t=0$, but with different external forces $h,h'\in\Sigma$. Then the function $\theta(t)=u_1(t)-u_2(t)$ satisfies
\begin{equation}\label{dif1}
\Dt^{2}\theta+\gamma\Dt\theta-\Delta \theta+l(t)\theta=h-h', \ \ \xi_\theta\big|_{t=0}=0.
\end{equation}
Then we have the following result which is immediately from \eqref{embed1}.
\begin{corollary}\label{difreg}
Let the assumptions of Proposition \ref{compattr} hold. 
Then the second ingredient $\Dt \theta$ of the solution $\xi_\theta$ for equation \eqref{dif1} satisfying that
$$
\Dt \theta\in L_b^{\infty}(T,+\infty;H_0^{1})\cap H_b^{1}(T,+\infty;H^{-1})
$$
and
\begin{equation}\label{dif2reg}
\|\Dt \theta\|_{L_b^{\infty}(T,+\infty;H_0^{1})\cap H_b^{1}(T,+\infty;H^{-1})}\le r,
\end{equation}
where $T$ is defined in \eqref{T} and the positive constant $r$ only depends on $\|g\|_W$.
\end{corollary}

Denote the set
\begin{gather*}
\begin{split}
A=
\begin{pmatrix}
A_{1}\\
A_{2}\\
\end{pmatrix}
=\{\xi_{\theta}|\xi_{\theta}(s),~s>0~\text{satisfies}~\eqref{dif1}~\text{with} ~\xi_{\theta}(0)=0,~h,h'\in\Sigma\}.
\end{split}
\end{gather*}
According to Corollary \ref{difreg},
we infer that $A\subset L_{b}^{\infty}(T,+\infty;\mathcal{E}^{1})$. 
From equation \eqref{dif1} we know that 
$\{\partial_{t}^{2}\theta|\partial_{t}\theta\in A_{2}\}\subset L_{b}^{2}(T,+\infty;H^{-1}(\Omega))$. That is,
$$A_{2}\subset L_{b}^{\infty}(T,+\infty;H_0^{1}(\Omega))\cap H_{b}^{1}(T,+\infty;H^{-1}(\Omega)).$$
Denote by
\begin{equation}\label{XY}
X:=L^{\infty}(iT,(i+1)T;H_0^{1})\cap H^{1}(iT,(i+1)T;H^{-1}) \ \ \text{and} \ \ Y:=L^{2}(iT,(i+1)T;L^{2}).
\end{equation}
Then for any $i\in\N^{+}$,
$$A_2\big|_{[iT,(i+1)T]}\subset B_r:=\{u\in X;\|u\|_X\le r\} \ \ \text{and} \ \  \Sigma_{[iT,(i+1)T]}\subset B_s:=\{v\in Y;\|v\|_Y\le s\},$$
where the positive constants $r$ given in \eqref{dif2reg} and $s=\|g\|_{L_{b}^{2}(\mathbb{R};L^{2}(\Omega))}$.

Thus, from Section $3$, for any given $\tilde{\varepsilon}>0$, we obtain a $(r,X;\|g\|_{L_{b}^{2}(\mathbb{R};L^{2}(\Omega))},Y;\tilde{\varepsilon})$-weak partition of $\Sigma_{[iT,(i+1)T]}$, $i\in\Bbb N^{+}$. Here we denote by $\Sigma_{[iT,(i+1)T]}\subset\cup_{n\in\Lambda}V_n^{(i)}$, $i\in\N^{+}$ and $N_{\tilde{\eb}}^{w}(r,X;\|g\|_{L_{b}^{2}(\mathbb{R};L^{2}(\Omega))},Y)$ be the minimal number of the cardinality of $\Lambda$. Then, for any given $l\in\N^{+}$, we will construct a special cover of $\Sigma_{T,(l+1)T}$ in the following way such that
\begin{equation}\label{partitionl}
\begin{split}
&\Sigma_{[T,(l+1)T]}\subset\cup_{j=1}^{M}W_j;\\
&\Sigma_{[iT,(i+1)T]}\cap W_j\big|_{[iT,(i+1)T]}\subset\cup_{n\in\Lambda}V_n^{(i)}, \ \ j=1,\cdots, M, \ i=1,\cdots,l.
\end{split}
\end{equation}
By combination the partitions on each subinterval $[iT,(i+1)T]$, $i=1,\cdots,l$, we obtain a cover of set $\Sigma_{[T,(l+1)T]}$, and the upper bound of the number of such cover can be estimated by
\begin{equation}\label{covernum}
M\leq(N_{\tilde{\varepsilon}}^w(r,X;\,\|g\|_{L_{b}^{2}(\mathbb{R};L^{2}(\Omega))},Y))^{l}.
\end{equation}

\begin{lemma}\label{lem5.5}
For every $\varepsilon>0$ and any given $l\in\mathbb{N^{+}}$, for the cover of $\Sigma_{[T,(l+1)T]}$ constructed in \eqref{partitionl}, we have that if $\tilde{\varepsilon}\leq\varepsilon^{2}/32e^{l(K+1)}$, then for any $h,h'\in W_{j}$, $j=1,\cdots,M$,
\begin{equation}\label{eq5.22}
dist_{\mathcal{E}}(\mathcal{K}_{h}((l+1)T),U_{h'}((l+1)T,T)\mathcal{K}_{h}(T))\leq\varepsilon/4.
\end{equation}
\end{lemma}
\begin{proof}
If $y\in\mathcal{K}_{h}((l+1)T)$, then we have $y=\xi((l+1)T)=U_{h}((l+1)T,T)\xi(T)$ for some $\xi\in\mathcal{K}_{h}$. Consider
\begin{equation*}
\|U_{h}((l+1)T,T)\xi(T)-U_{h'}((l+1)T,T)\xi(T)\|_{\mathcal{E}}.
\end{equation*}
Let $\xi_{u_{1}}(t)=U_{h}(t,T)\xi(T)$ and $\xi_{u_{2}}(t)=U_{h'}(t,T)\xi(T)$ for $t\geq T$. Then $\theta=u_{1}-u_{2}$ satisfying the equation \eqref{dif1}.
Multiplying \eqref{dif1} by $\partial_{t}\theta$, using \eqref{eq1.2} we get
\begin{equation*}
\frac{d}{dt}\|\xi_{\theta}(t)\|_{\mathcal{E}}^{2}+2\gamma\|\partial_{t}\theta(t)\|^{2}\leq C\|\xi_{\theta}\|_{\mathcal{E}}^{2}(1+\|u_{1}\|_{L^{12}}^{4}+\|u_{2}\|_{L^{12}}^{4})+2(h-h',\partial_{t}\theta).
\end{equation*}
Applying the Gronwall inequality and let $t=(l+1)T$, from \eqref{partitionl} and estimate \eqref{eq2.2} we obtain
\begin{equation*}
\begin{split}
\|\xi_{\theta}((l+1)T)\|_{\mathcal{E}}^{2}&\leq2\int_{T}^{(l+1)T}e^{C\int_{s}^{(l+1)T} (1+\|u_{1}\|_{L^{12}}^{4}+\|u_{2}\|_{L^{12}}^{4})dr}(h-h',\partial_{t}\theta)ds\\
&\leq2e^{C\int_{T}^{(l+1)T}(1+\|u_{1}\|_{L^{12}}^{4}+\|u_{2}\|_{L^{12}}^{4})dr}\sum\limits_{i=1}^{l}\int_{iT}^{(i
+1)T}(h-h',\partial_{t}\theta)ds\\
&\leq2e^{lKT}l\tilde{\varepsilon}\leq2e^{l(K+1)T}\tilde{\varepsilon}\leq(\varepsilon/4)^{2}
\end{split}
\end{equation*}
provided that
\begin{equation*}
\tilde{\varepsilon}\leq\frac{\varepsilon^{2}}{32e^{l(K+1)T}},
\end{equation*}
where
\begin{equation}\label{K}
K:=C\int_t^{t+1}(1+\|u_{1}(s)\|_{L^{12}}^{4}+\|u_{2}(s)\|_{L^{12}}^{4})ds\le Q(\|g\|_{L^{2}_{b}(\mathbb{R};L^{2}(\Omega))}).
\end{equation}
\end{proof}

\begin{remark}
Note that we find an upper bound for the number of a partition that covers the set $\Sigma_{[T,(l+1)T]}$, which may not be a sharp upper bound, but it depends only on the spaces $X$, $Y$ and the size $r$ given in \eqref{dif2reg} and $s=\|g\|_{L_{b}^{2}(\mathbb{R};L^{2}(\Omega))}$.
\end{remark}

Finally, we obtain the following theorem which is the main result of this paper.
\begin{theorem}\label{thm4.1}
Let the assumptions \eqref{eq1.2}-\eqref{eq1.4} be valid and $g$ satisfy either \eqref{gt-reg} or \eqref{gs-reg}. Then the process generated by \eqref{eq1.1} has a compact uniform attractor $\Cal{A}_{\Sigma}$, and for any arbitrary and fixed $\delta>0$,
there exist $T>0$ and $\varepsilon_{0}>0$ such that
\begin{equation*}
\begin{split}
\mathbb{H}_{\varepsilon}(\Cal A_{\Sigma})\leq(\mathcal{N}+\delta)\log_{2}\frac{2}{\varepsilon}+(\frac{1}{\sigma}\ln\frac{4\nu}{\varepsilon}+T)\mathbb{H}_{\frac{\varepsilon^{[2+(K+1)/\sigma]}}{32e^{T(K+1)}(4\nu)^{(K+1)/\sigma}+1}}^{w}(r,X;\,\|g\|_{L_{b}^{2}(\mathbb{R};L^{2}(\Omega))},Y)
\end{split}
\end{equation*}
for all $\varepsilon\leq\varepsilon_{0}$, where $\mathcal{N}$, $\nu$
and $\sigma>0$ satisfying \eqref{eq5.9} and \eqref{eq5.10}, the positive constants $r$ and $K$ are given in \eqref{dif2reg} and \eqref{K} respectively, and all of them depend only on the parameters of equation \eqref{eq1.1} and $\|g\|_W$; the spaces $X$ and $Y$ are defined in \eqref{XY}.
\end{theorem}

\begin{proof}
According to \eqref{partitionl}, we set
\begin{equation*}
\Sigma_{j}:=\Sigma\cap W_{j}, \ \ \ j=1,\cdots,M.
\end{equation*}
The representation \eqref{1.rep} can be written as
\begin{equation}\label{eq5.24}
\Cal{A}_{\Sigma}=\cup_{j=1}^{M}\mathcal{K}_{\Sigma_{j}}(t),\ \ \ t\in\mathbb{R},
\end{equation}
where $\mathcal{K}_{\Sigma_{j}}:=\cup_{h\in\Sigma_{j}}\mathcal{K}_{h}$.

Since the estimates \eqref{eq5.9} and \eqref{eq5.10} is uniformly for $h\in\Sigma_{0}$, then there exists $k=k(\varepsilon)>\frac{1}{\sigma T}\ln\frac{4\nu}{\varepsilon}$ and $k\in\mathbb{N}$, here $T$ is given in \eqref{T} and we take
\begin{equation}\label{eq5.25}
k=[\frac{1}{\sigma T}\ln\frac{4\nu}{\varepsilon}]+1
\end{equation}
such that for any $h_{j}\in\Sigma_{0}\cap \Sigma_{j}$,
\begin{equation}\label{eq5.26}
dist_{\mathcal{E}}(U_{h_{j}}((k+1)T,T)\Cal B_{1},\mathcal{M}_{h_{j}}^{0}(k+1))\leq \nu e^{-\sigma kT}<\varepsilon/4,
\end{equation}
and for every $\delta>0$, there exists $\varepsilon_{0}>0$ such that
\begin{equation}\label{eq5.27}
N_{\varepsilon/2}(\mathcal{M}_{h_{j}}^{0}(k+1);\mathcal{E})\leq (2/\varepsilon)^{\mathcal{N}+\delta}, \ \ \text{for all}~0<\varepsilon\leq\varepsilon_{0}.
\end{equation}
Combining with the estimates \eqref{eq5.22} for $l=k$ and \eqref{eq5.26}, we obtain that
\begin{equation*}
\begin{split}
&dist_{\mathcal{E}}(\mathcal{K}_{h}((k+1)T),\mathcal{M}_{h_{j}}^{0}(k+1))\leq dist_{\mathcal{E}}(\mathcal{K}_{h}((k+1)T),U_{h_{j}}((k+1)T,T)\mathcal{K}_{h}(T))\\
&\ \ \ \ \ \ \ \ \ \ \ \ \ \ \ \ \ \ \ \ \ \ \ \ \ \ \ \ \ \ \ \ \ \ \ \ \ \ \ \ \ \ \ \ \ \ +dist_{\mathcal{E}}(U_{h_{j}}((k+1)T,T)\mathcal{K}_{h}(T),\mathcal{M}_{h_{j}}^{0}(k+1))\\
&\ \ \ \ \ \ \ \ \ \ \ \ \ \ \ \ \ \ \ \ \ \ \ \ \ \ \ \ \ \ \ \ \ \ \ \ \ \ \ \ \ \ \ \leq\varepsilon/4+\varepsilon/4=\varepsilon/2, \ \ \ \forall~h\in\Sigma_{j}.
\end{split}
\end{equation*}
That is,
\begin{equation}\label{*4}
\begin{split}
&\mathcal{K}_{h}((k+1)T)\subset\mathbb{O}_{\varepsilon/2}(\mathcal{M}_{h_{j}}^{0}(k+1);\mathcal{E}), \ \ \forall~h\in\Sigma_{j}.
\end{split}
\end{equation}
Thus from \eqref{eq5.27} and \eqref{*4}, for all $0<\varepsilon\leq\varepsilon_{0}$, we have
\begin{equation*}
\begin{split}
N_{\varepsilon}(\mathcal{K}_{\Sigma_{j}}((k+1)T);\mathcal{E})\leq N_{\varepsilon}(\mathbb{O}_{\varepsilon/2}(\mathcal{M}_{h_{j}}^{0}(k+1);\mathcal{E});\mathcal{E}) =N_{\varepsilon/2}(\mathcal{M}_{h_{j}}^{0}(k+1);\mathcal{E})\leq(2/\varepsilon)^{\mathcal{N}+\delta}.
\end{split}
\end{equation*}
Using \eqref{eq5.24} for $t=(k+1)T$, let $\tilde{\varepsilon}=\frac{\varepsilon^{2}}{32e^{kT(K+1)}+1}$ and from \eqref{covernum}, \eqref{eq5.25}, we find that
\begin{equation*}
\begin{split}
N_{\varepsilon}(\Cal{A}_{\Sigma};\mathcal{E})&\leq  N_{\varepsilon}(\mathcal{K}_{\Sigma_{j}}((k+1)T);\mathcal{E})M\\
& \leq (2/\varepsilon)^{\mathcal{N}+\delta} \cdot (N_{\varepsilon^{2}/(32e^{kT(K+1)}+1)}^{w}(r,X;\,\|g\|_{L_{b}^{2}(\mathbb{R};L^{2}(\Omega))},Y))^{k}, ~~\forall~0<\varepsilon\leq\varepsilon_{0}.
\end{split}
\end{equation*}
Hence the Kolmogorov $\varepsilon$-entropy $\mathbb{H}_{\varepsilon}(\Cal{A}_{\Sigma})=\log_{2}N_{\varepsilon}(\Cal{A}_{\Sigma};\mathcal{E})$ of $\Cal{A}_{\Sigma}$ satisfies the estimate
\begin{equation*}
\begin{split}
\mathbb{H}_{\varepsilon}(\Cal{A}_{\Sigma})\leq(\mathcal{N}+\delta)\log_{2}\frac{2}{\varepsilon}+(\frac{1}{\sigma}\ln\frac{4\nu}{\varepsilon}+T)\mathbb{H}_{\frac{\varepsilon^{[2+(K+1)/\sigma]}}{32e^{T(K+1)}(4\nu)^{(K+1)/\sigma}+1}}^{w}(r,X;\,\|g\|_{L_{b}^{2}(\mathbb{R};L^{2}(\Omega))},Y).
\end{split}
\end{equation*}
\end{proof}

\end{document}